\documentclass{amsart}
\usepackage[demo]{graphicx}
\usepackage{caption}
\usepackage{subcaption}
\usepackage{graphicx, amsmath, amsfonts, epsfig, latexsym, amssymb, amsthm}
\usepackage{tikz}
\usepackage{tkz-berge}
\usepackage{float}

\usetikzlibrary {positioning}
\definecolor {processblue}{cmyk}{0.96,0,0,0}
\newtheorem{thm}{Theorem}[section]
\theoremstyle{definition}
\newtheorem{cor}[thm]{Corollary}
\newtheorem{prop}[thm]{Proposition}
\newtheorem{defn}[thm]{Definition}
\newtheorem{lem}[thm]{Lemma}

\newtheorem{no}[thm]{Notation}
\newtheorem{rem}[thm]{Remark}
\newtheorem{ex}[thm]{Example}

\newtheorem{co}[thm]{Question}
\numberwithin{equation}{section}
\begin{document}
\title[Ideal-based quasi cozero divisor graph of a commutative ring]
{Ideal-based quasi cozero divisor graph of a commutative ring}
\author{ F. Farshadifar}
\address{Department of Mathematics Education, Farhangian University, P.O. Box 14665-889, Tehran, Iran.}
\email{f.farshadifar@cfu.ac.ir}


\begin{abstract}
Let $R$ be a commutative ring with
identity and $I$ be an ideal of $R$.
The zero-divisor graph of $R$ with respect to $I$, denoted by $\Gamma_I(R)$, is the graph whose vertices are the set $\{ x \in R\setminus I\, |\, xy \in I\: for\: some\: y \in R \setminus I\}$ with distinct vertices $x$ and $y$ are adjacent if and only if $xy \in I$.
The cozero-divisor graph with respect to $I$, denoted by $\Gamma''_I(R)$, is the
graph of $R$ with vertices $\{x \in R\setminus I\: |\: xR+I\not=R \}$ and two distinct vertices $x$ and $y$ are adjacent if and only if  $x \not \in yR+I$ and $y \not \in xR+I$.
In this paper, we introduced and investigated an undirected graph $Q\Gamma''_I(R)$ of $R$ with vertices $\{x \in R\setminus \sqrt{I}\: |\: xR+I\not=R \ and \ xR+\sqrt{I}=xR+I\}$ and two distinct vertices $x$ and $y$ are adjacent if and only if $x \not \in yR+I$ and $y \not \in xR+I$.
\end{abstract}
\maketitle
\section{Introduction}
\noindent
Throughout this paper, $R$ will denote a commutative ring with identity, $\Bbb Z$, and $\Bbb Z_n$ will denote the rings of of integers and  integers modulo $n$, respectively. Also, the set of maximal ideals and the Jacobson radical of $R$ will denote by $Max(R)$ and $J(R)$, respectively.
For an ideal $I$ of $R$, let $\mathcal{M}(I)$ be the set of maximal ideals of $R$ containing
$I$ and the intersection of all maximal ideals of $R$ containing $I$ is denote by $J_I(R)$.

A \emph{graph} $G$ is defined as the pair $(V(G),E(G))$, where $V(G)$ is the set of vertices of $G$ and $E(G)$ is the set of edges of $G$. For two distinct vertices $a$ and $b$ of $V(G)$, the notation $a-b$ means that $a$ and $b$ are adjacent. A graph $G$ is said to be \emph{complete} if $a-b$ for all distinct $a, b\in V(G)$, and $G$ is said to be \emph{empty} if $E(G) =\emptyset$. Note by this definition that a graph may be empty even if $V (G)\not =\emptyset$. An empty graph could also be described as totally disconnected. If $|V (G)|\geq 2$, a \emph{path} from $a$ to $b$ is a series
of adjacent vertices $a-v_1-v_2-...-v_n-b$. The \emph{length of a path} is the number of edges it contains. A \emph{cycle} is a path that begins and ends at the same vertex in which no edge is repeated, and all vertices other than the starting and ending vertex are distinct. If a graph $G$ has a cycle, the \emph{girth} of $G$ (notated $g(G)$) is defined as the length of the shortest cycle of $G$; otherwise, $g(G) =\infty$. A graph $G$ is \emph{connected} if for every pair of distinct vertices $a, b\in V (G)$, there exists a path from $a$ to $b$. If there is a path from $a$ to $b$ with $a, b \in V (G)$,
then the \emph{distance from} $a$ to $b$ is the length of the shortest path from $a$ to $b$ and is denoted by $d(a, b)$. If there is not a path between $a$ and $b$, $d(a, b) = \infty$. The \emph{diameter} of $G$ is diam$(G) = Sup\{d(a,b)| a, b \in V(G)\}$.
A graph is said to be \textit{planar} if it can be drawn in the plane so that its edges intersect only at their ends.

Let $Z(R)$ be the set of all zero-divisors of $R$. Anderson and Livingston, in \cite{2}, introduced the \emph{zero-divisor graph of R}, denoted by $\Gamma(R)$, as the (undirected) graph with vertices $Z^*(R) = Z(R)\backslash \{0\}$ and for two distinct elements $x$ and $y$ in $Z^*(R)$, the vertices $x$ and $y$ are adjacent if and only if $xy = 0$.
In \cite{3}, Redmond introduced the definition of the zero-divisor graph with respect to an ideal. Let $I$ be an ideal of $R$. The \emph{zero-divisor graph of $R$ with respect to $I$}, denoted by $\Gamma_I(R)$, is the graph whose vertices are the set $\{ x \in R\setminus I\, |\, xy \in I\ for\ some\ y \in R \setminus I\}$ with distinct vertices $x$ and $y$ are adjacent if and only if $xy \in I$. Thus if $I = 0$, then $\Gamma_I(R) = \Gamma(R)$.

In \cite{1}, Afkhami and Khashayarmanesh introduced the \emph{cozero-divisor graph} $\Gamma'(R)$ of $R$, in which the vertices are precisely the non-zero, non-unit elements of $R$, denoted by $W^*(R)$, and two distinct vertices $x$ and $y$ are adjacent if and only if $x \not \in yR$ and $y \not \in xR$. Let $I$ be an ideal of $R$. Farshadifar in \cite{403}, introduced and studied a generalization of cozero-divisor graph with respect to $I$, denoted by $\Gamma''_I(R)$, which is
an undirected graph with vertices $\{x \in R\setminus I\: |\: xR+I\not=R \}$ and two distinct vertices $x$ and $y$ are adjacent if and only if  $x \not \in yR+I$ and $y \not \in xR+I$.
In fact, $\Gamma''_I(R)$ can be regarded as a dual notion of ideal-based zero-divisor graph introduced in \cite{3} and also,
$\Gamma''_I(R)$ is a generalization of cozero-divisor graph introduced in \cite{1} when $I = R$.

Let $I$ be a proper ideal of $R$. The \textit{ideal-based quasi zero divisor graph $Q\Gamma_I(R)$ of $R$ with respect to $I$} which is
an undirected graph with vertex set $\{ x \in R\setminus \sqrt{I}\, |\, xy \in I\: for\: some\: y \in R \setminus \sqrt{I}\}$ with two distinct vertices $x$ and $y$ are adjacent if and only if $xy \in I$ \cite{MR4351492}.

In this paper, we introduce and study an undirected graph $Q\Gamma''_I(R)$ of $R$ with vertices $\{x \in R\setminus \sqrt{I}\: |\: xR+I\not=R \ and \ xR+\sqrt{I}=xR+I\}$ and two distinct vertices $x$ and $y$ are adjacent if and only if $x \not \in yR+I$ and $y \not \in xR+I$, where $I$ is an ideal of $R$.
\section{Main Results}
\begin{defn}\label{2.1}
Let $I$ be an ideal of $R$. We define an undirected graph $Q\Gamma''_I(R)$ of $R$ with vertices $\{x \in R\setminus \sqrt{I}\: |\: xR+I\not=R \ and \ xR+\sqrt{I}=xR+I\}$. The distinct vertices $x$ and $y$ are adjacent if and only if $x \not \in yR+I$ and $y \not \in xR+I$. This can be regarded as a dual notion of ideal-based quasi zero-divisor graph introduced in \cite{MR4351492}.
\end{defn}

\begin{rem}\label{66.1}
\begin{itemize}
\item [(a)] Clearly, for an ideal $I$ of $R$, we have $xR+\sqrt{I}\not=R$ if and only if $xR+I\not=R$ for each $x \in R$.
\item [(b)] Clearly, if $I$ is an ideal of $R$ such that $\sqrt{I}$ is a maximal ideal of $R$ or $I=R$, then $Q\Gamma''_I(R) = \emptyset$. So, in the rest of this paper $\sqrt{I}$ is a non-maximal ideal of $R$ and $I \not=R$.
\end{itemize}
\end{rem}

\begin{prop}\label{2.4}
Let $I$ be an ideal of $R$. Then we have the following.
\begin{itemize}
\item [(a)] If $R/I$ is a reduced ring (or equivalently, if
$\sqrt{I}=I$), then $Q\Gamma''_I(R)=\Gamma''_I(R)$.
\item [(b)] $Q\Gamma''_I(R)$ is a subgraph of $\Gamma''_{\sqrt{I}}(R)$.
\item [(c)] $\Gamma''_{\sqrt{I}}(R)$ is an induced subgraph of $\Gamma''_I(R)$.
\item [(d)] If $J$ is an ideal of $R$, then $V(Q\Gamma''_J(R))\setminus T_I\subseteq V(Q\Gamma''_I(R))$. Moreover,
$Q\Gamma''_I(R)\setminus T_I$ is an induced subgraph of $Q\Gamma''_I(R)$.
\end{itemize}
 \end{prop}
\begin{proof}
These are clear.
 \end{proof}

 \begin{ex}\label{2.2}
Let $R=\Bbb Z_{12}$ and $I=0$. Then in the Figures (A), (B), and (C)  we can see that the graph
$Q\Gamma''_I(R)$ is an induced subgraph of $\Gamma''_{\sqrt{I}}(R)$ and the graph $\Gamma''_{\sqrt{I}}(R)$ is an induced subgraph of $\Gamma''_I(R)$.
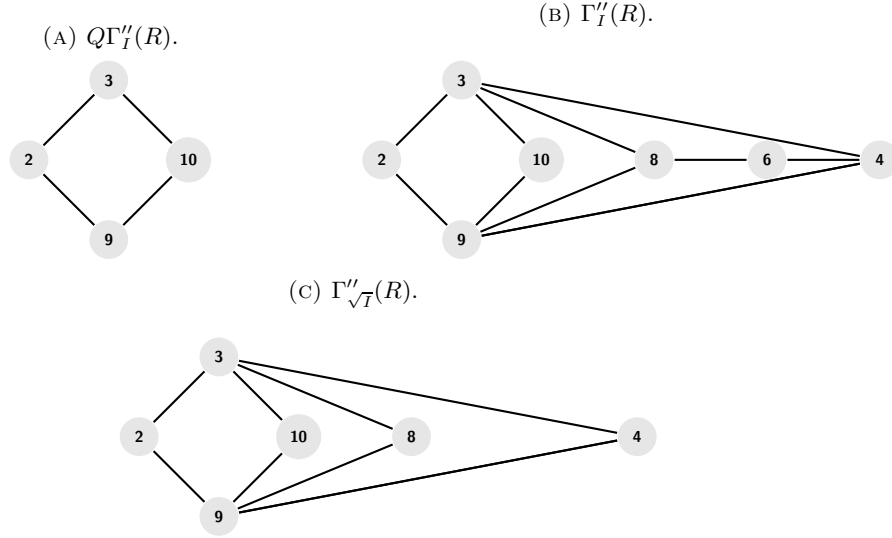
\begin{figure}[H]
\centering
\begin{subfigure}[b]{0.49\textwidth}
\centering
\caption{$Q\Gamma''_I(R)$.}
\begin{center}
\begin{tikzpicture}[auto,node distance=1.5cm,
  thick,main node/.style={circle,fill=black!10,font=\sffamily\tiny\bfseries}]
\node[main node] (1) {3};
\node[main node] (2)[below left of=1] {2};
\node[main node] (3)[below right of=1] {10};
\node[main node] (7) [below right of=2] {9};
\path[every node/.style={font=\sffamily\small}]
    (1) edge node [left] {} (2)
     (1) edge node [left] {} (3)
              (7) edge node [left] {} (2)
         (7) edge node [left] {} (3);
\end{tikzpicture}
\end{center}
\end{subfigure}
\hfill
\begin{subfigure}[b]{0.49\textwidth}
\centering
\caption{$\Gamma''_I(R)$.}
\begin{center}
\begin{tikzpicture}[auto,node distance=1.5cm,
  thick,main node/.style={circle,fill=black!10,font=\sffamily\tiny\bfseries}]
\node[main node] (1) {3};
\node[main node] (2)[below left of=1] {2};
\node[main node] (3)[below right of=1] {10};
\node[main node] (4) [right of=3] {8};
\node[main node] (5) [right of=4] {6};
\node[main node] (6) [right of=5] {4};
\node[main node] (7) [below right of=2] {9};
\path[every node/.style={font=\sffamily\small}]
    (1) edge node [left] {} (2)
     (1) edge node [left] {} (3)
      (1) edge node [left] {} (4)
       (7) edge node [left] {} (6)
        (7) edge node [left] {} (2)
         (7) edge node [left] {} (3)
          (7) edge node [left] {} (4)
           (7) edge node [left] {} (6)
            (1) edge node [left] {} (6)
             (5) edge node [left] {} (4)
                 (5) edge node [left] {} (6);
\end{tikzpicture}
\end{center}
\end{subfigure}
\hfill
\begin{subfigure}[b]{0.49\textwidth}
\centering
\caption{$\Gamma''_{\sqrt{I}}(R)$.}
\begin{center}
\begin{tikzpicture}[auto,node distance=1.5cm,
  thick,main node/.style={circle,fill=black!10,font=\sffamily\tiny\bfseries}]
\node[main node] (1) {3};
\node[main node] (2)[below left of=1] {2};
\node[main node] (3)[below right of=1] {10};
\node[main node] (4) [right of=3] {8};
\node[main node] (6) [right of=5] {4};
\node[main node] (7) [below right of=2] {9};
\path[every node/.style={font=\sffamily\small}]
    (1) edge node [left] {} (2)
     (1) edge node [left] {} (3)
      (1) edge node [left] {} (4)
       (7) edge node [left] {} (6)
        (7) edge node [left] {} (2)
         (7) edge node [left] {} (3)
          (7) edge node [left] {} (4)
           (7) edge node [left] {} (6)
            (1) edge node [left] {} (6);

\end{tikzpicture}
\end{center}
\end{subfigure}
\end{figure}
 \end{ex}

 An ideal $I$ of $R$ is said to be a \emph{secondary ideal} if $I \neq 0$ and for every element $r$ of $R$ we have either $r\in \sqrt{I}$ or $rI= I$.
\begin{lem}\label{second}
$R/I$ is a secondary ideal of $R/I$ if and only if $Q\Gamma''_I(R) = \emptyset$.
\begin{proof}
Straightforward.
\end{proof}
\end{lem}

 \begin{prop}\label{2.5}
Let $I$ be a non-zero proper ideal of $R$. Then $Q\Gamma''_I(R)$ is not a
complete graph.
 \end{prop}
\begin{proof}
Assume that $x$ is a vertex of $Q\Gamma''_I(R)$. It is
clear that $x + i\not= x$ is also a vertex of $Q\Gamma''_I(R)$, where $0\not= i \in I$. Now, $x+i \in Rx+I$ implies that $x$ is not adjacent to $x+i$. Thus $Q\Gamma''_I(R)$ is not complete.
\end{proof}

\begin{thm}\label{2.6}
Let $n=p^{\alpha}q^{\beta}$, where $p$ and $q$ are distinct prime numbers and consider $R=\Bbb Z_n$. Then for every proper ideal $I$ of $R$, we have $\Gamma''_{\sqrt{I}(R)}=\emptyset$ or $\Gamma''_{\sqrt{I}}(R)$ is a complete bipartite graph.
\end{thm}
\begin{proof}
If $I=p^{\alpha_1}\Bbb Z_n$ or $I=q^{\beta_1}\Bbb Z_n$, where $\alpha_1\leq \alpha$ and $\beta_1\leq \beta$, then $\Gamma''_{\sqrt{I}(R)}=\emptyset$. Otherwise, if $I=p^{\alpha_1}q^{\beta_1}\Bbb Z_n$, $0\not=\alpha_1\leq \alpha$ and $0\not=\beta_1\leq \beta$, then $V(\Gamma''_{\sqrt{I}}(R))=V_1 \cup V_2$, where $V_1=p\Bbb Z_n\setminus (p \cap q)\Bbb Z_n$ and
$V_2=q\Bbb Z_n\setminus (p \cap q)\Bbb Z_n$. Clearly, $V_1$, $V_2$ are independent sets and any vertex in $V_1$ is adjacent to any arbitrary vertex in $V_2$. Therefore, $\Gamma''_{\sqrt{I}}(R)$ is a complete bipartite graph.
\end{proof}

 \begin{ex}\label{2.3}
Let $R=\Bbb Z_{24}$ and $I=12\Bbb Z_{24}$. Then in the Figures (D), (E), and (F)  we can see that the graph
$Q\Gamma''_I(R)$ is an induced subgraph of $\Gamma''_{\sqrt{I}}(R)$ and the graph $\Gamma''_{\sqrt{I}}(R)$ is an induced subgraph of $\Gamma''_I(R)$.
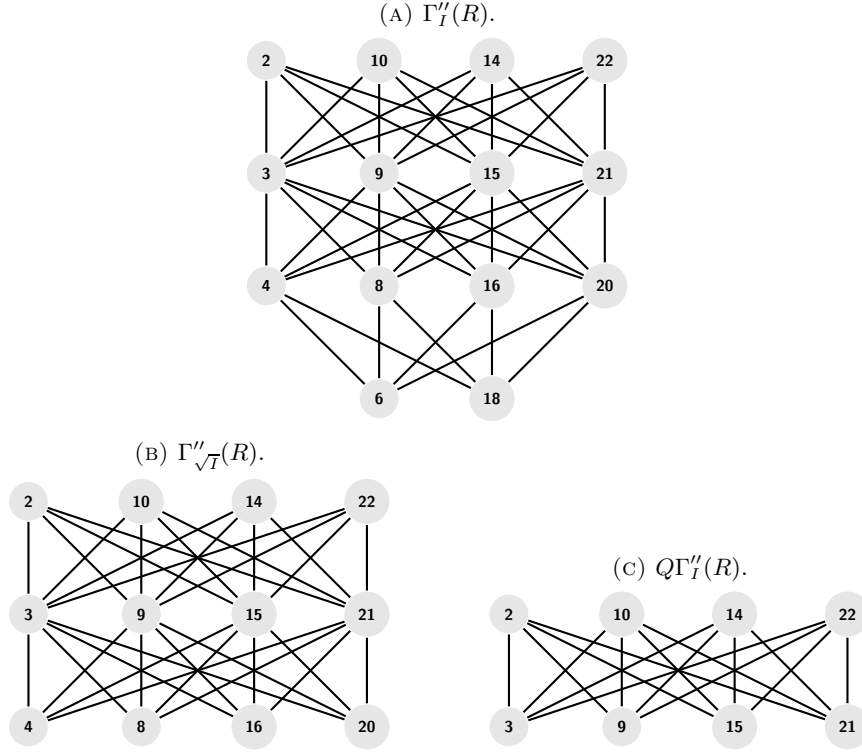
\begin{figure}[H]
\centering
\begin{subfigure}[h]{0.49\textwidth}
\centering
\caption{$\Gamma''_I(R)$.}
\begin{center}
\begin{tikzpicture}[auto,node distance=1.5cm,
  thick,main node/.style={circle,fill=black!10,font=\sffamily\tiny\bfseries}]
\node[main node] (1) {2};
\node[main node] (2) [right of=1] {10};
\node[main node] (3) [right of=2] {14};
\node[main node] (4) [right of=3] {22};
\node[main node] (5) [below of=1] {3};
\node[main node] (6) [right of=5] {9};
\node[main node] (7) [right of=6] {15};
\node[main node] (8) [right of=7] {21};
\node[main node] (9) [below of=5] {4};
\node[main node] (10) [right of=9] {8};
\node[main node] (11) [right of=10] {16};
\node[main node] (12) [right of=11] {20};
\node[main node] (13) [below of=10] {6};
\node[main node] (14) [right of=13] {18};
\path[every node/.style={font=\sffamily\small}]
    (1) edge node [left] {} (5)
    (1) edge node [left] {} (6)
        (1) edge node [left] {} (7)
         (1) edge node [left] {} (8)
          (2) edge node [left] {} (5)
           (2) edge node [left] {} (6)
            (2) edge node [left] {} (7)
             (2) edge node [left] {} (8)
              (3) edge node [left] {} (5)
               (3) edge node [left] {} (6)
                (3) edge node [left] {} (7)
                 (3) edge node [left] {} (8)
                  (4) edge node [left] {} (5)
                   (4) edge node [left] {} (6)
                    (4) edge node [left] {} (7)
                     (4) edge node [left] {} (8)
                      (9) edge node [left] {} (14)
                       (9) edge node [left] {} (5)
                        (9) edge node [left] {} (6)
                         (9) edge node [left] {} (7)
                          (9) edge node [left] {} (8)
                         (9) edge node [left] {} (13)
                         (10) edge node [left] {} (5)
       (10) edge node [left] {} (6)
        (10) edge node [left] {} (7)
         (10) edge node [left] {} (8)
          (11) edge node [left] {} (5)
           (11) edge node [left] {} (6)
            (11) edge node [left] {} (7)
             (11) edge node [left] {} (8)
              (12) edge node [left] {} (5)
               (12) edge node [left] {} (6)
                (12) edge node [left] {} (7)
                 (12) edge node [left] {} (8)
                  (13) edge node [left] {} (10)
                   (13) edge node [left] {} (11)
                    (13) edge node [left] {} (12)
            (14) edge node [left] {} (10)
                   (14) edge node [left] {} (11)
                    (14) edge node [left] {} (12);
\end{tikzpicture}
\end{center}
\end{subfigure}

\hfill
\begin{subfigure}[b]{0.49\textwidth}
\centering
\caption{$\Gamma''_{\sqrt{I}}(R)$.}
\begin{center}
\begin{tikzpicture}[auto,node distance=1.5cm,
  thick,main node/.style={circle,fill=black!10,font=\sffamily\tiny\bfseries}]
\node[main node] (1) {2};
\node[main node] (2) [right of=1] {10};
\node[main node] (3) [right of=2] {14};
\node[main node] (4) [right of=3] {22};
\node[main node] (5) [below of=1] {3};
\node[main node] (6) [right of=5] {9};
\node[main node] (7) [right of=6] {15};
\node[main node] (8) [right of=7] {21};
\node[main node] (9) [below of=5] {4};
\node[main node] (10) [right of=9] {8};
\node[main node] (11) [right of=10] {16};
\node[main node] (12) [right of=11] {20};
\path[every node/.style={font=\sffamily\small}]
    (1) edge node [left] {} (5)
    (1) edge node [left] {} (6)
        (1) edge node [left] {} (7)
         (1) edge node [left] {} (8)
          (2) edge node [left] {} (5)
           (2) edge node [left] {} (6)
            (2) edge node [left] {} (7)
             (2) edge node [left] {} (8)
              (3) edge node [left] {} (5)
               (3) edge node [left] {} (6)
                (3) edge node [left] {} (7)
                 (3) edge node [left] {} (8)
                  (4) edge node [left] {} (5)
                   (4) edge node [left] {} (6)
                    (4) edge node [left] {} (7)
                     (4) edge node [left] {} (8)
                                            (9) edge node [left] {} (5)
                        (9) edge node [left] {} (6)
                         (9) edge node [left] {} (7)
                          (9) edge node [left] {} (8)
                                               (10) edge node [left] {} (5)
       (10) edge node [left] {} (6)
        (10) edge node [left] {} (7)
         (10) edge node [left] {} (8)
          (11) edge node [left] {} (5)
           (11) edge node [left] {} (6)
            (11) edge node [left] {} (7)
             (11) edge node [left] {} (8)
              (12) edge node [left] {} (5)
               (12) edge node [left] {} (6)
                (12) edge node [left] {} (7)
                 (12) edge node [left] {} (8);
                  \end{tikzpicture}
\end{center}
\end{subfigure}
\hfill
\begin{subfigure}[b]{0.49\textwidth}
\centering
\caption{$Q\Gamma''_I(R)$.}
\begin{center}
\begin{tikzpicture}[auto,node distance=1.5cm,
  thick,main node/.style={circle,fill=black!10,font=\sffamily\tiny\bfseries}]
\node[main node] (1) {2};
\node[main node] (2) [right of=1] {10};
\node[main node] (3) [right of=2] {14};
\node[main node] (4) [right of=3] {22};
\node[main node] (5) [below of=2] {9};
\node[main node] (7) [right of=5] {15};
\node[main node] (8) [right of=7] {21};
\node[main node] (9) [below of=1] {3};
\path[every node/.style={font=\sffamily\small}]
    (1) edge node [left] {} (5)
        (1) edge node [left] {} (7)
         (1) edge node [left] {} (8)
          (2) edge node [left] {} (5)
                      (2) edge node [left] {} (7)
             (2) edge node [left] {} (8)
              (3) edge node [left] {} (5)
                               (3) edge node [left] {} (7)
                 (3) edge node [left] {} (8)
                  (4) edge node [left] {} (5)
                                     (4) edge node [left] {} (7)
                                                (9) edge node [left] {} (1)
           (9) edge node [left] {} (2)
           (9) edge node [left] {} (3)
           (9) edge node [left] {} (4)
                     (4) edge node [left] {} (8);
\end{tikzpicture}
\end{center}
\end{subfigure}
\end{figure}
 \end{ex}

Let $G$ be a graph. A non-empty subset $D$ of the
vertex set $V(G)$ is called a \textit{dominating set} if every vertex $V (G\setminus D)$ is adjacent to at least
one vertex of $D$. The \textit{domination number} $\gamma(G)$ is the minimum cardinality among the dominating sets of $G$.
\begin{thm}\label{2.7}
Let $n= p_1^{\alpha_1}p_2^{\alpha_2}\cdots p_k^{\alpha_k}$,
where $p_i’s$ are distinct primes and $k > 1$.
Consider the vertices $x_i = n/p_i^{\alpha_i}$ for $i = 1, 2, \ldots , k$ of the graph $\Gamma''_{n\Bbb Z}(\Bbb Z)\setminus \sqrt{n\Bbb Z}$. Then
$S = \{x_i : i = 1, 2, \ldots , k\}$ is a dominating set for $\Gamma''_{n\Bbb Z}(\Bbb Z)\setminus \sqrt{n\Bbb Z}$ and so
the domination number $\gamma$ of $\Gamma''_{n\Bbb Z}(\Bbb Z)\setminus \sqrt{n\Bbb Z}$ is $k$.
\end{thm}
\begin{proof}
Let $x$ be an arbitrary vertex in
$ \Gamma''_{n\Bbb Z}(\Bbb Z)\setminus \sqrt{n\Bbb Z}$. Then $\sqrt{n\Bbb Z}=p_1p_2\cdots p_k$ does not divide $x$ and there exists $j \in \{1, 2, \ldots k\}$ such that $p_j^{\alpha_t}$ ($\alpha_t \leq \alpha_j$) divide $x$. Observe that $x\not \in x_j\Bbb Z+n\Bbb Z$ and $x_j\not \in x\Bbb Z+n\Bbb Z$, i.e., $x$ is adjacent to $x_j$. Hence, $S$ is a dominating
set and hence  $\gamma \leq k$.
Now, assume contrary that $\gamma < k$. Then there exists a dominating set $S'$ with $k - 1$ vertices. Let
$S' =\{y_1, y_2, \ldots, y_{k-1}\}$. Consider the set of vertices $D =\{p_1^{\alpha_1},p_2^{\alpha_2},\ldots, p_k^{\alpha_k}\}$. If any
$p_i^{\alpha_i} \in S'$, then we replace $p_i^{\alpha_i}$ in $D$ by $pp_i^{\alpha_i}$, where $p$ is a prime which does not divide $n$
and $pp_i^{\alpha_i}\not \in  S'$. This can be guaranteed as choice of such a $p$ is infinite. Thus $D\cap  S'=\emptyset$.
Since  $S'$ is a dominating set, each element of $D$ is adjacent to some element of  $S'$. We
claim that two distinct elements of $p_i^{\alpha_i}$ and $p_j^{\alpha_j}$ of $D$ can not be dominated by same $y_t$.
Because, if it happens, then $p_i^{\alpha_i}\not \in y_t\Bbb Z+n\Bbb Z$ and $p_j^{\alpha_j}\not \in y_t\Bbb Z+n\Bbb Z$, i.e., both $n/p_i^{\alpha_i}$ and $n/p_j^{\alpha_j}$ divides $y_t$,
i.e., their l.c.m. divides $y_t$, i.e., $n|y_t$, i.e., $y_t \in n\Bbb Z$, a contradiction. Therefore, distinct
$p_i^{\alpha_i}$`s are dominated by distinct elements of $S'$ and hence $S'$ should contain at least $k$
vertices, a contradiction. Thus $\gamma= k$ and we are done.
\end{proof}

\begin{cor}\label{2.8}
Let $I = n\Bbb Z$ be an ideal of the ring $\Bbb Z$. Then we have the following.
\begin{itemize}
\item [(a)] If $n = 0$ or $n= p^k$, where $p$ is prime and $k$ is a positive integer, then $\Gamma''_{n\Bbb Z}(\Bbb Z)\setminus \sqrt{n\Bbb Z}$ is
an empty graph.
\item [(b)] If $n = p_1^{\alpha_1}p_2^{\alpha_2}$, where $p_1$, $p_2$ are distinct primes, then $\Gamma''_{n\Bbb Z}(\Bbb Z)\setminus \sqrt{n\Bbb Z}$ is a complete bipartite graph with diam($\Gamma''_{n\Bbb Z}(\Bbb Z)\setminus \sqrt{n\Bbb Z})=2$ and gr($\Gamma''_{n\Bbb Z}(\Bbb Z)\setminus \sqrt{n\Bbb Z})=4$.
\item [(c)]  $n= p_1^{\alpha_1}p_2^{\alpha_2}\cdots p_k^{\alpha_k}$,
where $p_i’s$ are distinct primes and $k >2$, then $\Gamma''_{n\Bbb Z}(\Bbb Z)\setminus \sqrt{n\Bbb Z}$ is
a $k$-partite graph with the domination number  $\gamma= k$.
\end{itemize}
\end{cor}

\begin{thm}\label{2.12}
Let $I$ be an ideal of $R$. Then we have the following.
 \begin{itemize}
   \item [(a)] The graph $\Gamma''_I(R)\setminus J_I(R)$ is connected.
   \item [(b)] If $\vert \mathcal{M}(I)\vert\not=1$, then diam $(\Gamma''_I(R)\setminus J_I(R)) \leq 2$.
 \end{itemize}
 \end{thm}
\begin{proof}
(a) If  $\vert \mathcal{M}(I)\vert=1$, then $V(\Gamma''_I(R)) \setminus J_I(R)$ is the empty set; which is connected. So we may assume that $\vert \mathcal{M}(I)\vert > 1$. Let $x,y \in V(\Gamma''_I(R)) \setminus J_I(R)$ be two distinct elements. Without loss of generality, we may assume that $x \in yR+I$. Since $x \not\in J_I(R)$, there exists $\mathfrak{m}\in \mathcal{M}(I)$ such that $x \not\in \mathfrak{m}$. We claim that $\mathfrak{m} \nsubseteq J_I(R)\cup (yR+I)$. Otherwise, $\mathfrak{m} \subseteq J_I(R) \cup (yR+I)$. This implies that $\mathfrak{m} \subseteq J_I(R)$ or $\mathfrak{m} \subseteq yR+I$. But $\mathfrak{m} \neq J_I(R)$. Hence we have $\mathfrak{m}  \subseteq yR+I \varsubsetneq R$, so $\mathfrak{m} =yR+I$. Thus $x \in \mathfrak{m}$, a contradiction. Choose the element $z \in \mathfrak{m}\setminus J_I(R) \cup (yR+I)$. If $y \in zR+I$, then $yR+I \subseteq zR+I$. Thus $x \in yR+I\subseteq zR+I\subseteq \mathfrak{m}  +I=\mathfrak{m}$. This contradiction shows that $z-y$. Now as $xR+I\subseteq yR+I$ and $zR+I\subseteq\mathfrak{m}+I=\mathfrak{m}$ we have $z \not \in xR+I$ and $x \not\in zR+I$. Therefore, $x-z-y$.

(b) This follows from part (a).
\end{proof}
The\textit{chromatic number} of a graph $G$, denoted by $\chi(G)$, is the minimal number of colors which
can be assigned to the vertices of $G$ in such a way that every two adjacent vertices have
different colors. Also, a \textit{clique of a graph} is a complete subgraph and the number of vertices
in a largest clique of $G$, denoted by $\omega(G)$, is called the \textit{clique number} of $G$.
A graph $H$ is called a \textit{retract} of $G$ if there are homomorphisms $\rho: G \rightarrow H$ and
$\varphi : H \rightarrow G$ such that $\rho \circ \varphi= id_H$. The homomorphism $\rho$ is called a \textit{retraction} \cite{HT97}.
If $H$ is a retract of $G$, then chromatic number and clique number of $G$ and $H$ are same \cite[Observation 2.17]{HT97}.
\begin{thm}\label{2.9}
Let $I$ be an ideal of $R$. Then $\Gamma'(R/I)$  is a retract of $\Gamma''_I(R)$.
\end{thm}
\begin{proof}
Define the map $\rho: V (\Gamma''_I(R)) \rightarrow V (\Gamma'(R/I))$ by $\rho(x) = x + I$. For
each coset $x + I \in  V (\Gamma'(R/I)))$, choose and fix a representative $x^* \in x + I$ and define
$\varphi : V (\Gamma'(R/I)) \rightarrow V (\Gamma''_I(R))$ by $\varphi(x + I) = x^*$. Then it is easy to see that $\rho$ is a surjective graph homomorphism and $\varphi$ is a graph homomorphism.
Moreover, $\rho \circ \varphi : V (\Gamma'(R/I)) \rightarrow V (\Gamma''_I(R))$ is given by $(\rho \circ \varphi)(x + I) = \rho(x^*) =
x^* + I = x + I$. Thus $\rho \circ \varphi=id_{\Gamma'(R/I)}$ and so $\Gamma'(R/I)$ is a retract
of $\Gamma''_I(R)$.
\end{proof}

\begin{cor}\label{2.10}
Let $I$ be an ideal of $R$. Then the graphs $\Gamma'(R/I)$ and $\Gamma''_I(R)$ have same chromatic number and clique number.
\end{cor}

\begin{thm}\label{2.11}
Let $I$ be an ideal of $R$ and $a, b \in R\setminus \sqrt{I}$ such that $a+I=a+\sqrt{I}$ and $b+I=b+\sqrt{I}$. Then the following
statements hold.
\begin{itemize}
\item [(a)] If $a + I$ is adjacent to $b + I$ in $\Gamma'(R/I)$, then $a$ is adjacent to $b$ in $Q\Gamma''_I(R)$.
\item [(b)] If $a$ is adjacent to $b$ in $Q\Gamma''_I(R)$, then $a+\sqrt{I}$ and $b+\sqrt{I}$ are always distinct elements,
and also they are adjacent in $\Gamma'(R/\sqrt{I})$. Moreover, $Q\Gamma''_I(R)$ is isomorphic to a subgraph of $\Gamma'(R/\sqrt{I})$.
\end{itemize}
\end{thm}
\begin{proof}
(a) Let $a+I$ is adjacent to $b+I$ in $\Gamma'(R/I)$. Then $a+I\not \in (b+I)(R/I)$ and $b+I\not \in (a+I)(R/I)$. Thus  $a\not \in (b+I)R+I$ and $b\not \in (a+I)(R)+I$. This implies that
$a\not \in bR+I=(b+0)R+I\subseteq (b+I)R+I$ and $b\not \in aR+I=(a+0)R+I\subseteq (a+I)R+I$. By assumption, $a, b \in R\setminus \sqrt{I}$, $a+I=a+\sqrt{I}$, and $b+I=b+\sqrt{I}$. Thus $a$ is adjacent to $b$ in $Q\Gamma''_I(R)$.

(b) Suppose that $a-b$ in $Q\Gamma''_I(R)$. Assume on the contrary that $a+\sqrt{I}=b+\sqrt{I}$. Then $a+I=b+I$ since by assumption, $a+I=a+\sqrt{I}$ and $b+I=b+\sqrt{I}$. This implies that $a  \in bR+I$. This contradiction shows that $a+\sqrt{I}\not=b+\sqrt{I}$.
Now, if $a+I \in (b+I)(R/I)$, then $a\in (b+I)R+I\subseteq bR+I$, a contradiction. Thus  $a+I \not \in (b+I)(R/I)$. Similarly,
$b+I\not \in (a+I)(R/I)$ and so $a+\sqrt{I}$ and $b+\sqrt{I}$ are adjacent in $\Gamma'(R/\sqrt{I})$.
Now, we show that $Q\Gamma''_I(R)$
is isomorphic to a subgraph of $\Gamma'(R/\sqrt{I})$. We define a graph $G$ with vertices $\{a_i : a_i +\sqrt{I} \ is \ a  \ vertex \ of  \  \Gamma'(R/\sqrt{I})\}$, where $a_i - a_j$ if  $a_i \not \in a_jR+I$ and $a_j \not \in a_iR+I$. Then $G$ is a subgraph of
$\Gamma'(R/\sqrt{I})$.
\end{proof}

\begin{defn}\label{2.13}
We say that an ideal $I$ of $R$ is a \textit{sum-radical ideal} if $\sqrt{I}=I+Rx$ for some $x \in \sqrt{I}$.
\end{defn}

Recall that a \textit{radical ideal} of $R$ is an ideal that is equal to its radical.
Clearly, every radical ideal is a sum-radical ideal. But as we can see in the following example the converse is not true in general.
\begin{ex}\label{2.14}
Consider the ideal $4\Bbb Z$ of the ring $\Bbb Z$. Then $4\Bbb Z$ is not a radical ideal. But $2\in \sqrt{4\Bbb Z}=2\Bbb Z$ and
$2\Bbb Z+4\Bbb Z=2\Bbb Z$ implies that $4\Bbb Z$  is a sum-radical ideal of $\Bbb Z$. Similarly, one can see that every ideal of the ring $\Bbb Z$ is a sum-radical ideal.
\end{ex}

\begin{no}\label{nnn}
For an ideal $I$ of $R$, we set $T_I:=\{x \in R\, |\, Rx+I\not=Rx+\sqrt{I}\}$.
\end{no}

\begin{lem}\label{2.15}
Let $I$ be an ideal of $R$. Then we have the following.
\begin{itemize}
\item [(a)]  $V(Q\Gamma_I(R)) \cap T_I\subseteq V(Q\Gamma''_I(R))$.
\item [(b)] If $I$ is a sum-radical ideal of $R$, then $\mathfrak{m} \not \subseteq T_I$ for each $\mathfrak{m}\in \mathcal{M}(I)$.
\end{itemize}
\end{lem}
\begin{proof}
(a) This is straightforward.

(b) Let $I$ be a sum-radical ideal of $R$. Then we have $\sqrt{I}=I+Rx$ for some $x \in \sqrt{I}$. Thus $\sqrt{I}+Rx=I+Rx$ and so $x \not \in T_I$.
Assume contrary that  $\mathfrak{m} \subseteq T_I$ for some $\mathfrak{m}\in \mathcal{M}(I)$. Then
$I\subseteq \mathfrak{m} \subseteq T_I$. It follows that $\sqrt{I}\subseteq \mathfrak{m} \subseteq T_I$ and so $x \in T_I$, a contradiction.
\end{proof}

Similar to the Theorem \ref{2.12} it is natural to ask the following
question:
\begin{co}\label{qq}
Let $I$ be an ideal of $R$. Is the graph $Q\Gamma''_I(R)\setminus J_I(R)$ connected?
\end{co}

The following theorem gives an affirmative answer to the Question \ref{qq}, when $I$ is a sum-radical ideal of $R$.
\begin{thm}\label{2.16}
Let $I$ be a sum-radical ideal of $R$. Then we have the following.
 \begin{itemize}
   \item [(a)] The graph $Q\Gamma''_I(R)\setminus J_I(R)$ is connected.
   \item [(b)] If $\vert \mathcal{M}(I)\vert\not=1$, then diam $(Q\Gamma''_I(R)\setminus J_I(R)) \leq 2$.
 \end{itemize}
 \end{thm}
\begin{proof}
(a) If $\vert \mathcal{M}(I)\vert=1$, then $J_I(R)=\mathfrak{m}$ for some $\mathfrak{m} \in Max(R)$. If $x \in V(Q\Gamma''_I(R))$, then $xR+I \not=R$ implies that $x \in \mathfrak{m}=J_I(R)$. Therefore,
 $V(\Gamma''_I(R)) \setminus J_I(R)$ is the empty set; which is connected. So we may assume that $\vert \mathcal{M}(I)\vert > 1$. Let $x,y \in V(\Gamma''_I(R)) \setminus J_I(R)$ be two distinct elements. Without loss of generality, we may assume that $x \in yR+I$ and so $x \in yR+\sqrt{I}$ since $yR+I=yR+\sqrt{I}$. As $x \not\in J_I(R)$, there exists $\mathfrak{m}\in \mathcal{M}(I)$ such that $x \not\in \mathfrak{m}$. We claim that $\mathfrak{m} \nsubseteq J_I(R)\cup (yR+\sqrt{I})\cup T_I$. Otherwise, $\mathfrak{m} \subseteq J_I(R) \cup (yR+\sqrt{I})\cup T_I$. It follows that $\mathfrak{m} \subseteq J_I(R)$ or $\mathfrak{m} \subseteq yR+\sqrt{I}$ or $\mathfrak{m} \subseteq T_I$. But $\mathfrak{m} \neq J_I(R)$ and since $I$ is a sum-radical ideal,  $\mathfrak{m} \not \subseteq T_I$ by Lemma \ref{2.15} (b). Hence we have $\mathfrak{m}  \subseteq yR+\sqrt{I} \varsubsetneq R$, so $\mathfrak{m} =yR+\sqrt{I}$. This implies that $x \in \mathfrak{m}$, a contradiction. Choose the element $z \in \mathfrak{m}\setminus J_I(R) \cup (yR+\sqrt{I})\cup T_I$. Then $z \in V(Q\Gamma''_I(R))$.  If $y \in zR+I$, then $yR+I \subseteq zR+I$. Thus $x \in yR+I\subseteq zR+I\subseteq \mathfrak{m}  +I=\mathfrak{m}$. This contradiction shows that $z-y$. Now since $xR+I\subseteq yR+I\subseteq yR+\sqrt{I}$ and $zR+I\subseteq\mathfrak{m}+I=\mathfrak{m}$ we have $z \not \in xR+I$ and $x \not\in zR+I$. Therefore, $x-z-y$.

(b) This follows from part (a).
\end{proof}

\begin{cor}\label{43.44}
Let $R$ be a non-local ring and $I$ be a sum-radical ideal of $R$. Then $g(\Gamma''_I(R) \setminus J_I(R)) \leq 5$ or $g(\Gamma''_I(R) \setminus J_I(R)) = \infty$.
\end{cor}
\begin{proof}
This follows from Theorem \ref{2.16} and the technique of
 \cite[Theorem 2.8]{1}.
\end{proof}

\begin{thm}\label{29.7}
Let $I$ be a sum-radical ideal of $R$ such that $|\mathcal{M}(I)|\geq 3$. Then $g(Q\Gamma''_I(R)) = 3$.
\end{thm}
\begin{proof}
As $\Gamma''_I(R)$ is a simple graphs, $g(\Gamma''_I(R)) \geq 3$. Assume that $\mathfrak{m_1}$, $\mathfrak{m_2}$ and
$\mathfrak{m_3}$ are distinct maximal ideals of $R$ such that $I \subseteq \mathfrak{m_i}$ for $i=1, 2, 3$. By using Lemma \ref{2.15} (b), we can choose $x_i \in \mathfrak{m_i}\setminus \bigcup^3_{j= 1, j\not=i}\mathfrak{m_j}\cup T_I$, for $i=1, 2, 3$. Then we have $x_i \in V(\Gamma''_I(R))$ for $i=1, 2, 3$ and $x_1 -x_2 - x_3 - x_1$. Thus  $g(\Gamma''_I(R)) = 3$.
\end{proof}

\begin{thm}\label{2.53}
Let $I$ be a sum-radical ideal of $R$ such that $\vert \mathcal{M}(I)\vert \geq 5$. Then $Q\Gamma''_I(R)$ is not planar.
\end{thm}
\begin{proof}
Let $\mathfrak{m_1}, . . . ,\mathfrak{m_5}$ be distinct maximal ideals of $R$ such that $I \subseteq \mathfrak{m_i}$ and $x_i\in \mathfrak{m_i}\setminus \bigcup^5_{j=1,j\not=i}\mathfrak{m_j}\bigcup T_I$ for $i = 1, . . . , 5$ by using Lemma \ref{2.15} (b). Then $x_1, x_2, . . . , x_5$ forms a complete subgraph
of $Q\Gamma''_I(R)$ which is isomorphic to $K_5$ and so, by \cite[p.153]{BM76}, $\Gamma''_I(R)$ is not planar.
\end{proof}

Recall that a vertex $x$ of a connected graph $G$ is said to be a \textit{cut-point} of $G$ if
there are vertices $u$, $w$ of $G$ such that $x$ is in every path from $u$ to $w$ (and
$x \not= u$, $x \not= w$). Equivalently, for a connected graph $G$, $x$ is a cut-point of
$G$ if $G- \{x\}$ is not connected \cite{3}.

\begin{thm}\label{2.56}
Let $I$ be a sum-radical ideal of $R$. Then $Q\Gamma''_I(R)\setminus J_I(R)$ has no cut-points.
\end{thm}
\begin{proof}
Assume contrary that the vertex $x$ of $Q\Gamma''_I(R)\setminus J_I(R)$ is a cut-point. Then there exist vertices
$u, w\in V(Q\Gamma''_I(R)\setminus J_I(R))$ such that $x$ lies on every path from $u$ to $w$. By Theorem \ref{2.16}, the shortest path from $u$ to $w$ is of length 2. Suppose $u-x-w$ is a path of shortest length from $u$ to $w$.
Then one can see that $x+i \in V(Q\Gamma''_I(R))\setminus J_I(R)$ and $u-(x+i)-w$ is a path in  $Q\Gamma''_I(R)\setminus J_I(R)$, which is a contradiction.
\end{proof}

Recall that a graph is \textit{Hamiltonian} if it contains a cycle which visits each vertex exactly
once and also returns to the starting vertex.
\begin{thm}\label{2.56997}
Let $I$ be a sum-radical ideal of a finite ring $R$. If $R$ have exactly two maximal ideals $\mathfrak{m_1}$ and $\mathfrak{m_2}$ such that $\vert  \mathfrak{m_1}\vert=\vert\mathfrak{m_2} \vert$ and $I \subseteq \mathfrak{m_i}$ for $i=1, 2$. Then $Q\Gamma''_I(R)\setminus J_I(R)$ is Hamiltonian.
\end{thm}
\begin{proof}
For $i = 1, 2$ put $\mathfrak{m_1} \setminus J_I(R)\cup T_I := \{a_{i1},...,a_{it}\}$, where $t :=\vert\mathfrak{m_1} \setminus J_I(R)\cup T_I \vert$. Then one can see that $a_{11}-a_{21}-...-a_{1t}-a_{2t}-a_{11}$ is a Hamiltonian cycle in $Q\Gamma''_I(R)\setminus J_I(R)$.
\end{proof}

\begin{thm}\label{2.56}
Let $I$ be a sum-radical ideal of $R$ and $\mathcal{M}(I) = \{\mathfrak{m_1}, \mathfrak{m_2}\}$. Then  $Q\Gamma''_I(R)\setminus J_I(R)$ is a complete bipartite graph with parts $\mathfrak{m_i}\setminus J(R)\cup T_I$ for $i = 1, 2$ if and only if for every $x, y \in \mathfrak{m_i}\setminus J_I(R)\cup T_I$, for some $i = 1, 2$,  the ideals $xR+I$ and $yR+I$ are totally ordered (i.e. either
$xR+I\subseteq yR+I$ or $yR+I\subseteq xR+I$).
\end{thm}
\begin{proof}
Let $Q\Gamma''_I(R)\setminus J_I(R)$ be a complete bipartite graph with parts $\mathfrak{m_i}\setminus J_I(R)\cup T_I$ for $i = 1, 2$.
Assume contrary that there exist $x, y \in \mathfrak{m_1}\setminus J_I(R)\cup T_I$ such that $xR+I\not\subseteq yR+I$ and $yR+I\not\subseteq xR+I$. Then $x$ is adjacent to $y$ in $\mathfrak{m_1}\setminus J_I(R)\cup T_I$,
which is a contradiction. Conversely, by assumption, for every elements $x$ and $y$  in $\mathfrak{m_i}\setminus J_I(R)\cup T_I$ for some $ i = 1, 2$, $x$ is not adjacent to $y$. Now, let $x \in \mathfrak{m_1}\setminus  \mathfrak{m_2}\cup T_I$ and  $y \in \mathfrak{m_2}\setminus  \mathfrak{m_1}\cup T_I$. If $x \in Ry+I$, then $x \in \mathfrak{m_2}+J_I(R)=\mathfrak{m_2}$. This contradiction shows that  $x\not \in Ry+I$. Similarly, $y\not \in Rx+I$ and so $x$ is adjacent to $y$. Thus $Q\Gamma''_I(R)\setminus J_I(R)$ is a complete bipartite graph with parts $\mathfrak{m_1}\setminus  \mathfrak{m_2}\cup T_I$ and $\mathfrak{m_2}\setminus  \mathfrak{m_1}\cup T_I$.
\end{proof}

\begin{prop}\label{2.57}
Let $I$ a sum-radical ideal of $R$ and the graph $Q\Gamma''_I(R)\setminus J_I(R)$ be $n$-partite for some positive integer $n$. Then $\vert \mathcal{M}(I)\vert\leq n$.
\end{prop}
\begin{proof}
Assume contrary that $\vert \mathcal{M}(I)\vert> n$. Since $Q\Gamma''_I(R)\setminus J_I(R)$ is an $n$-partite
graph, there are $\mathfrak{m_1}, \mathfrak{m_2} \in \mathcal{M}(I)$ such that there exist $x \in \mathfrak{m_1}\setminus  \mathfrak{m_2}\cup T_I$ and  $y \in \mathfrak{m_2}\setminus  \mathfrak{m_1}\cup T_I$ and $x,y$ belong to a same part. One can see that $x$ is adjacent to $y$,  which is a contradiction.
\end{proof}

\begin{thm}\label{2.5997}
Let $I$ a sum-radical ideal of $R$  such that $\vert \mathcal{M}(I)\vert\not=1$. Then the following conditions are equivalent:
\begin{itemize}
\item [(a)] $Q\Gamma''_I(R)\setminus J_I(R)$ is complete bipartite;
\item [(b)] $Q\Gamma''_I(R)\setminus J_I(R)$ is bipartite;
\item [(c)] $Q\Gamma''_I(R)\setminus J_(R)$ contains no triangles.
\end{itemize}
\end{thm}
\begin{proof}
Use the Proposition \ref{2.57} and the technique of \cite[Theorem 2.7]{1111}.
\end{proof}

\textbf{Conflicts of Interest.}
The author declares that there are no conflicts of interest.

\bibliographystyle{amsplain}

\begin{thebibliography}{10}
\bibitem{1}
M.~Afkhami and K.~Khashyarmanesh, \emph{The cozero-divisor graph of a commutative ring}, Southeast Asian Bull. Math., \textbf{35} (2011), 753--762.

\bibitem{1111}
M.~Afkhami and K.~Khashyarmanesh, \emph{On the Cozero-Divisor Graphs of Commutative Rings and Their Complements}, Bull. Malays. Math. Sci. Soc., (2) \textbf{35} (4) (2012), 935-944.

\bibitem{2}
D.F.~Anderson and P.S.~Livingston, \emph{The zero-divisor graph of a commutative ring}, J. Algebra, \textbf{217} (1999), 434--447.

\bibitem{BM76}
J. A Bondy and U.S.R. Murty, \emph{Graph Theory with Applications}, American Elsevier, New York, 1976.

\bibitem{MR4351492}
E.Y. \c{C}elikel, A. Das, and C. Abdio\u{g}lu,  \emph{Ideal-based quasi zero divisor graph}, Hacet. J. Math. Stat., \textbf{50} (6) (2021), 1658-1666.

\bibitem{403}
F. Farshadifar, \emph{A generalization of the cozero-divisor graph of a commutative ring}, Discrete Mathematics, Algorithms and Applications, to appear.

\bibitem{HT97}
G. Hahn and C. Tardif, \emph{Graph Homomorphisms: Structure and Symmetry}, in: Graph
Symmetry, 107-166, Springer, Dordrecht, 1997.

\bibitem{3}
S. P.~Redmond, \emph{An ideal-based zero-divisor graph of a commutative ring}, Comm. Algebra, \textbf{31} (2003), 4425--4443.

\bibitem{Y01}
S.~Yassemi, \emph{The dual notion of prime submodules}, Arch. Math (Brno),
\textbf{37} (2001), 273--278.
\end{thebibliography}

\end{document}